\documentclass[twoside,12pt]{article}

\usepackage{hyperref}
\usepackage[OT1]{fontenc}
\usepackage[a4paper,bottom=2cm,top=2cm,inner=3cm,outer=2.5cm,includehead]{geometry}

\usepackage{amssymb}
\usepackage{amsmath}
\usepackage{amsthm}

\usepackage{graphicx}
\usepackage{subfig}
\usepackage{dsfont}

\def\E{\mathds{E}}
\def\N{\mathds{N}}
\def\P{\mathds{P}}
\def\R{\mathds{R}}

\def\V{\mathds{V}}
\def\id{\mathds{1}}

\newcommand{\Per}{\mathop{\rm Per}\nolimits}
\newcommand{\Div}{\mathop{\rm div}\nolimits}

\def\bm #1{{\boldsymbol{#1}}}

\def\cB{{\mathcal B}}

\def\cF{{\mathcal F}}

\def\cH{{\mathcal H}}

\def\b{\beta}
\def\d{\delta}
\def\e{\varepsilon}

\def\k{{\kappa}}
\def\l{\lambda}

\def\s{\sigma}

\def\bd{{\partial}}

\newtheorem{theorem}{Theorem}[section]
\newtheorem{lemma}[theorem]{Lemma}

\begin{document}

\title{\bfseries Poisson-Delaunay approximation}

\author{Matthias Reitzner and Anna Strotmann}

\date{}

\maketitle

\begin{abstract}
For a Borel set $A$  and a stationary Poisson point process  $\eta_t$  in $\R^d$ of intensity $t>0$, the
Poisson-Delaunay approximation $ A_{\eta_t}$ of $A$  is the union of all Delaunay cells generated by $\eta_t$ with center in $A$. It is shown that $\l_d(A_{\eta_t})$ is an unbiased estimator for $\l_d(A)$, variance bounds and a quantitative central limit theorem are given. The asymptotic behaviour of the symmetric difference $\l_d(A\Delta A_\eta)$ is derived as $t \to\infty$.

\end{abstract}

\section{Introduction}

Triangulations play a central role in many methods of numerical data evaluation, for topographic maps, sensor technology, and much more, see e.g. \cite{topographic, sensor}. Due to their in many ways optimal properties the Delaunay triangulation is particularly suitable, see \cite{Rajan}. In this work we consider the Delaunay tessellation in $\R^d$ based on a random point set. Random tessellations like this are an important part in stochastic geometry and have applications in many different fields.

We use the random Delaunay tessellation to approximate an unknown set. Thinking of the simplices as pixels, there is an obvious connection to image analysis, quantization problems, and non-parametric statistics, see for example \cite{image}, \cite{GrLu} and \cite[Section~3]{EiKh}. And, clearly, approximating a set by random Delaunay simplices gives rise to a randomized numerical integration, estimating Lebesgue measure of an unknown set by the volume of the corresponding Delaunay simplices.

Let $\eta_t$ be a stationary Poisson point process in $\R^d$ with intensity $t > 0$.  A $(d+1)$-tuple of points in ${\eta_t}$ forms a Delaunay cell, defined as the convex hull of these points, if its circumball does not contain any further point from $\eta_t$. The collection $D_{\eta_t}$ of all Delaunay cells is called the Poisson-Delaunay tessellation of $\R^d$. The center $\bm c$ of a Delaunay cell is the center of the circumball. Let $C(D_{\eta_t})$ be the set of all centers of Delaunay cells $D_{\eta_t}(\bm c)$ of $\eta_t$, hence 
$$
D_{\eta_{t}}:=\bigcup_{\bm c  \in C(D_{\eta_t}) } D_{\eta_t}(\bm c) .
$$
Precise definitions are given in the next section.

For a Borel set $A \subset\R^d$, we define the Poisson-Delaunay approximation $A_{\eta_{t}}$ of $A$ as the union of all Delaunay cells with center inside of the set $A$ (see Figure \ref{approx}), 
\begin{equation}\label{def:PoisDel}
A_{\eta_{t}} :=\bigcup_{\bm c  \in C(D_{\eta_t}) \cap A} D_{\eta_t}(\bm c) .
\end{equation}
This yields a random approximation of the set $A$. The aim of this paper is to investigate how well this construction determines the volume of a Borel set and to measure the deviation. In particular, in Theorem \ref{thm:exp} we will see that $\E \l_d(A_{\eta_t})= \l_d(A)$ and thus $\l_d(A_{\eta_t})$ is an unbiased estimator of the possible unknown set $A$. In Theorems \ref{thm:var} and \ref{thm:CLT} we will state variance bounds and a central limit theorem for $\l_d(A_{\eta_t})$ as $t\to \infty$ if $A$ is convex. Finally, in Theorem \ref{thm:sym} we give bounds on the symmetric difference between $A$ and $A_{\eta_t}$.

\textbf{\begin{figure}[ht]
		\centering
		\includegraphics[scale=0.25]{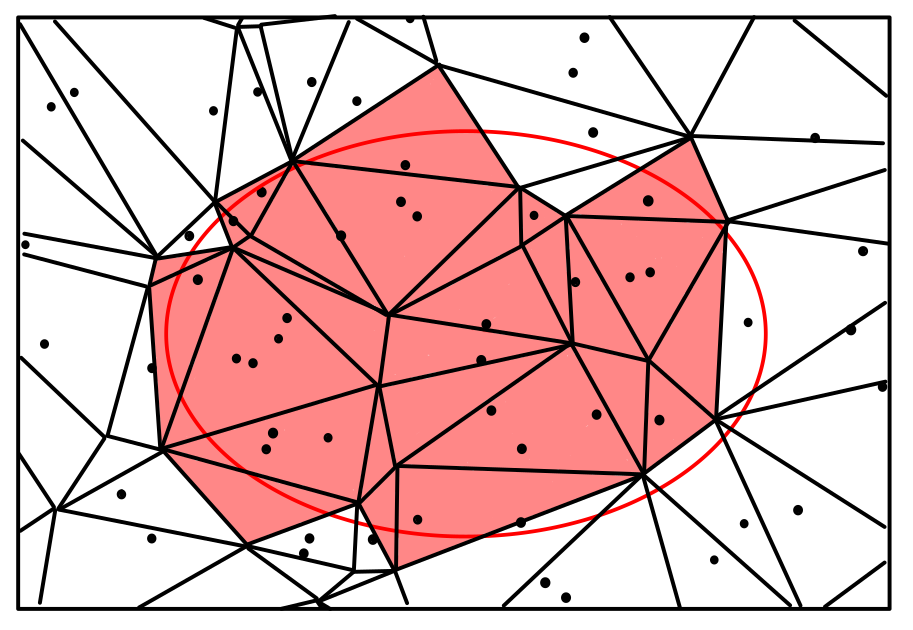}
		\caption{Poisson-Delaunay approximation of an ellipse.}
		\label{approx}
\end{figure}}

It is well known that the construction of the Delaunay tessellation is dual to the construction of the Voronoi tessellation. The Voronoi cell $V_{\eta_{t}}(\bm x)$ of a point $\bm x\in\eta_t$ is defined as the set of all points, which are closer to $\bm x$ than any other point of the  Poisson point process.
The union of all Voronoi cells $V_{\eta_{t}}(\bm x)$, $\bm x\in\eta_t$, is called the Poisson-Voronoi tessellation of $\R^d$, which we denote by $V_{\eta_t}$. Let ${\mathcal{F}_{0}(V_{\eta_{t}})}$ be the set of all vertices of the Voronoi tessellation. An alternative way of defining the Delaunay tessellation via the Voronoi cells is to define for a vertex $\bm c \in \mathcal{F}_{0}(V_{\eta_{t}})$ the associated Delaunay cell as the convex hull of the centers of those Voronoi cells which have the point $\bm c$ as a vertex. Since a Poisson-Voronoi mosaic is normal, i.e. every vertex belongs almost surely to $d+1$ cells, each Delaunay cell has exactly $d+1$ vertices. Therefore, it is almost surely a simplex. The Poisson-Delaunay tessellation thus is dual to the Poisson-Voronoi tessellation and the set of all centers $C(D_{\eta_t})$ satisfies $\mathcal{F}_{0}(V_{\eta_{t}})=C(D_{\eta_t})$. 

Quite similar to the approximation described above, one can approximate a Borel set $A$ using the Voronoi tessellation. Such an approximation, which is the union of all Voronoi cells with center inside of $A$, is called Poisson-Voronoi approximation. The quality of the Poisson-Voronoi approximation was already extensively investigated. There are several publications dealing with variance bounds, see \cite{HevRei}, \cite{Schulte}, \cite{ThYu} and \cite{Yu}, central limit theorems, i.e. in \cite{Schulte}, \cite{Yu} and \cite{LaSchuYu}, and concentration inequalities, see \cite{HevRei}. In addition, properties of the symmetric difference of the Poisson-Voronoi approximation were investigated for example in \cite{HevRei}, \cite{ReiSpoZap}, \cite{LaSchuYu} and \cite{Yu}. In this work we want to investigate the dual Poisson-Delaunay approximation which has been unexplored so far, and check whether it is equally suitable.

In the next section, we will present our results. This is followed by a section that lists the relevant notations and tools. Finally, the proofs of the main results are presented in Sections \eqref{Esti}-\eqref{Sym}.

\section{Main results}

A short computation shows that the Poisson-Delaunay approximation $A_{\eta_t}$ as well as the Voronoi approximation is an unbiased estimator for the volume of $A$.

\begin{theorem}
\label{thm:exp}
Let $A$ be a Borel set. Then 
$$
\E \l_d (A_{\eta_t})=\l_d(A).
$$
\end{theorem}

The next theorem is a quantitative central limit theorem, giving an upper bound for the Kolmogorov distance between the standardized Poisson-Delaunay approximation and a standard Gaussian random variable. 

\begin{theorem}
	\label{thm:CLT}
	Let $A$ be a compact convex set with interior points, and $N$ a standard Gaussian random variable. Then
	$$
	d_K\left(\frac{\l_d(A_{\eta_t}) -\E  \l_d(A_{\eta_t}) }{\sqrt{\V \l_d(A_{\eta_t}) }},N\right)\le c_{d,A} t^{-1/2+1/(2d)}.
	$$
\end{theorem} Since this bound converges to zero if the intensity $t$ of the Poisson point process tends to infinity, this yields a central limit theorem. 
The main tool of the proof is an abstract central limit theorem (see Theorem \ref{thm:clt}), developed by Lachi\'eze-Rey, Pecatti and Schulte, which is based on the Malliavin-Stein method and, in addition, makes use of stabilizing properties of the Delaunay tessellation. Moreover, the proof of Theorem \ref{thm:CLT} requires a lower variance bound.

\begin{theorem}
	\label{thm:var}
	Let $A$ be a compact convex set with interior points. There are constants $\overline{c_{d,A}}$ and $\underline{c_{d,A}}$ depending only on the dimension $d$ and the convex set $A$ such that
	\begin{equation*}
		\underline{c_{d,A}}t^{-1-\frac{1}{d}}\le\V  \l_d(A_{\eta_t})  \le \overline{c_{d,A}}t^{-1-\frac{1}{d}}
	\end{equation*}
	for $t\ge\left(\frac{8d}{r_A}\right)^d$, where $r_A$ is the inradius of $A$.
\end{theorem}
Both, the lower and the upper bound, are of the same order. While the computation for the upper bound is based on an application of the well known Poincar\'e inequality for Poisson point processes, the computation of the lower bound is a bit more involved, since it is quite difficult to control whether the difference operator of the Delaunay approximation is positive or negative. A helpful tool to work around this problem is a recent theorem by Schulte and Trapp \cite{SchuTra}, see Theorem \ref{thm:TraSchu}.

\textbf{\begin{figure}[ht]
		\centering
		\includegraphics[scale=0.25]{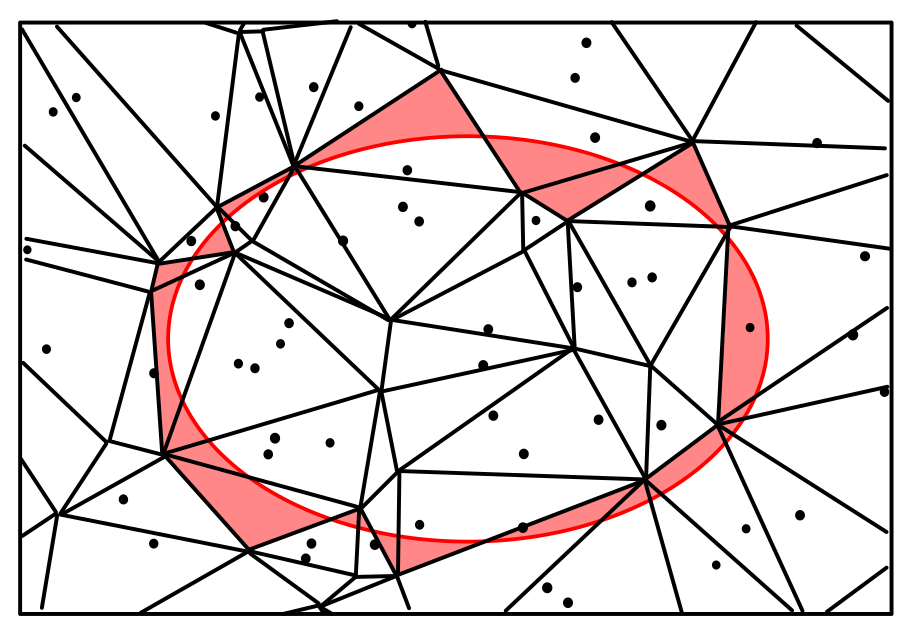}
		\caption{Symmetric difference of the Poisson-Delaunay approximation.}
		\label{Symmetric}
\end{figure}}

The volume of the symmetric difference of the Poisson-Delaunay approximation $A_{\eta_t}$ with $A$ is defined as the volume \lq on the false side of $A$\rq, $$\l_d(A_{\eta_t} \Delta A) = \l_d \big((A_{\eta_t} \backslash A) \cup (A \backslash A_{\eta_t})\big).$$ 

It is obvious that the expected volume of the symmetric difference decreases and tends to zero for growing intensity of the Poisson point process. In the following theorem we show that the rate of decay is $t^{- \frac 1d}$. 
\begin{theorem}\label{thm:sym}
Let $A\subset\R^d$ be a Borel set with $ \l_d (A)<\infty$ and $\Per (A)<\infty$. Then
	\begin{equation*}
\lim_{t \to \infty }t^{\frac 1d} \E \l_d (A_{\eta_t}\Delta A) = c_d\Per(A).
	\end{equation*}
\end{theorem}
Here $\Per(A)$ denotes the perimeter of $A$, which we will introduce in the next section. We just point out that for $A$ convex the perimeter corresponds to the $(d-1)$ dimensional Hausdorff measure of the boundary of $A$. There are upper and lower bounds for the constant $c_d$, see Section \eqref{Sym}. 

\bigskip
It turns out that Theorems \ref{thm:exp}--\ref{thm:sym} parallel results for the Poisson-Voronoi approximation, \cite{Schulte, HevRei, ThYu}. It has been shown by Th\"ale and Yukich \cite{ThYu}, that most results first shown for convex sets can be extended to so-called admissible sets. This class contains e.g. convex sets, smooth sets and sets of positive reach. This is also true for the investigations in this paper. We decided to restrict our proofs to compact convex sets in Theorems \ref{thm:CLT} and \ref{thm:var}, to keep the presentation concise.

\section{Preliminaries}
\subsection{Geometric basics}
The $d$-dimensional unit ball in $\R^d$ is denoted by $B_d$ and its boundary $ \bd B_d$ is the unit sphere $S^{d-1}$. Moreover, $B(\bm z ,r)$ is the $d$-dimensional ball with center in $\bm z \in\R^d$ and radius $r>0$.
The volume of the unit ball is given by 
$$ \k_d 
=
\frac{\pi^{\frac d2}}{\Gamma(\frac d2+1)},
$$
where $\Gamma (s)= \int\limits_0^\infty x^{s-1} e^{-x} dx$ is the Gamma function defined for $s >0$. The surface area of the unit sphere is $\omega_d=d\kappa_d$.
Integration of a function $f\colon \R^d \to \R$ with respect to Lebesgue measure is denoted by $\int f(\bm x) d\bm x$, resp. $\int f(x) dx$ for $d=1$.

For subsets $A,B\subset\R^d$ the \emph{Minkowski sum} is defined by
$A+B = \{ \bm a + \bm b  \colon \bm a  \in A, \bm b  \in B\}$. Furthermore, the dilate of a set $A$ by some $\e >0$ is the set 
$ \e A = \{ \e \bm a  \colon \bm a  \in A\}$. One of the famous classical results in the theory of convex bodies is the Steiner formula which says that for a convex body $K$ the volume of $K + \e \cB_d$ is a polynomial in $\e$. Let $K$ be a convex body and assume $\e\geq 0$, then 
\begin{equation}\label{eq:stein}
	\l_d(K+ \e \cB_d) =  \sum _{i=0}^ d \kappa_i V_{d-i}(K) \e^ i 
\end{equation}
with coefficients $V_i(K)$ only depending on $K$. We call these coefficients $V_i(K)$ the \emph{intrinsic volumes} of $K$. We see that the case $\e=0$ gives $V_d(K)= \l_d(K)$ and the case when $\e$ goes to infinity yields $V_0(K)=1$. Moreover, it holds $2V_{d-1}(K)=\cH^{d-1}(K)$, where $\cH^{d-1}(\cdot)$ is the $d-1$-dimensional Hausdorff measure.

\bigskip
Let
$\mathcal{C}_c^1(\R^d)$ be the class of all continuously
differentiable vector--valued functions from $\R^d$ to
$\R^d$ with compact support. The \textsl{perimeter} of a Borel set $A$ is defined as the total variation of the indicator function $\id(\cdot\in A)$ on $\R^d$,  
$$
\Per(A)=\sup\Big\{\int\limits_{A}\Div\varphi(\bm x)\, d\bm x\,:\,\varphi\in\mathcal{C}_c^1(\R^d),\|\varphi\|_\infty\leqslant1\Big\},
$$
where $$\Div\varphi(\bm x)=\sum_{i=1}^d\frac{\partial\varphi_i}{\partial x_i}$$ 
is the divergence operator and
$$\|\varphi\|_\infty=\max_{i=1,\dots,d}\sup_{\bm x\in\R^d}|\varphi_i(\bm x)|$$ 
denotes the supremum norm of $\varphi=(\varphi_1,\dots,\varphi_d)$. When $A$ is a compact
set with Lipschitz boundary, then $\Per(A)$ equals the $(d-1)$-dimensional Hausdorff measure $\cH^{d-1}(\partial A)$ of the boundary of $A$. This includes, as already mentioned, the case when $A$ is convex. 

For a Borel set $A$ with finite volume we define the \emph{covariogram} $g_A$ as 
$$
g_A(\bm x)=\l_d ((A+\bm x)\cap A) 
$$
for $x\in\R^d$.  
A result obtained by Galerne, see \cite[Theorem 14]{Gal}
gives us a useful connection between the perimeter and the covariogram.

\begin{theorem}
	\label{thm:per}
	Let $A$ be a Borel set with finite volume. Then the following assertions are equivalent:
	\begin{itemize}
		\item[(a)] $\Per(A)<\infty$;
		\item[(b)] there exists a finite limit
		\begin{equation*}\label{eq:1}
			\lim_{r\to+0}\frac{g_A(r \bm u)-g_A(0)}{r}  =  \frac{\partial g_A}{\partial  \bm u}(0)
		\end{equation*}
		for all $ \bm u\in S^{d-1}$;
		\item[(c)] $g_A$ is Lipschitz.
	\end{itemize}	
In addition,  the Lipschitz constant of $g_A$ satisfies $\mbox{\rm Lip}(g_A)\leqslant\frac12\Per(A)$ and it holds
	\begin{equation}\label{eq:2}
		\int\limits_{S^{d-1}}\frac{\partial g_A}{\partial  \bm u}(0) \, \mathcal{H}_{d-1} (d \bm u) = - \kappa_{d-1} \Per(A).
	\end{equation}
\end{theorem}
  
We will need an explicit simplex construction for the proof of the lower variance bound which we include here.
\begin{lemma}
	\label{le:simplex}
	Let $T$ be a regular simplex with vertices $\bm a _1, \dots, \bm a _{d+1}$, inradius $s$ and center $\bm c $.
	Then there is a $\d=\d(d)\in (0,\frac 15)$ such that the following holds: for each simplex $T'$ with vertices $\bm v _i \in B(\bm a _i,\d s)$, the center $\bm c '$ of the circumball of $T'$ is in $B(\bm c ,\frac 1{5}s)$, and each sphere $S$ of radius $R\geq 3 d^2 s $ containing exact $d$ of the vertices of $T'$ is disjoint from the ball $B(\bm c , \frac 45 s)$.
\end{lemma}
\begin{proof}
Clearly it suffices to consider the case $s=1$.
We start by investigating the situation for the regular simplex $T$ in $\R^d$.
The inradius of $T$ is $1$, the circumradius $d$, hence by Pythagoras' theorem the $(d-1)$-dimensional circumball of $\bm a _1, \dots, \bm a _{d}$ has radius $\sqrt{d^2- 1} $. Assume that $S$ contains $\bm a _1, \dots, \bm a _{d}$ and let $E$ be the affine hyperplane spanned by $\bm a _1, \dots, \bm a _{d}$. The height of the cap cut off from $S$ by $E$ is given by
$$
h
= R- \sqrt{R^2 -(d^2 - 1)}
$$
which is monotone decreasing in $R$.
Since $R\geq 3 d^2 s $, it holds
$$
h
\leq 3 d^2 - \sqrt{9d^4 -d^2 +\frac 1{36} }
= \frac 16.
$$
Because the inradius of $T$ is 1, this shows that this cap has distance at least $\frac 56$ to the center $\bm c $. We thus note the following observation:
\\[1ex]
{\it{Let $T$ be a regular simplex with vertices $\bm a _1, \dots, \bm a _{d+1}$, inradius $1$ and center $\bm c $. Then each sphere $S$ of radius $\geq 3 d^2$ containing exact $d$ of the vertices of $T$ is disjoint from the ball $B(\bm c , \frac 56)$.
}}
\\[1ex]
In the second step we note, that the position of the center of a simplex, and the height of a cap of a sphere cut off by a hyperplane are continuous functions in the coordinates of the vertices of the simplex. Hence if the vertices of $T'$ are sufficiently close to the vertices of $T$, e.g. $\|\bm v _i -\bm a _i\| \leq \d$ with $\d$ sufficiently small, then the center of the circumball $\bm c '$ is close to $\bm c $, $\| \bm c -\bm c '\| \leq \frac 15$, and if a sphere $S$ containing $d$ vertices of $T$ avoids $B(\bm c , \frac 56)$ then a sphere $S'$ containing $d$ vertices of $T'$ avoids $B(\bm c , \frac 45)$ for $\d$ sufficiently small.
\end{proof}

\subsection{Integral transformation}

First, we consider two integral transformation formulas of Blaschke-Petkantschin type. These formulas are based on the fact, that $d+1$ points in general position always lie on a uniquely defined sphere. So we can represent each point $ \bm x_i$ with the center $\bm c \in\R^d$ and the radius $r\ge0$ of the sphere combined with a directional vector $\bm u_i\in S^{d-1}$, i.e. $\bm x_i=\bm c +r \bm u_i$.  Let ${f:(\R^d)^{d+1}\to\R}$ be a nonnegative measurable function. Due to \cite[Theorem 7.2.2]{StoGeo} it holds
\begin{align} \label{eq:Bla}
	\int\limits_{(\R^d)^{d+1}}f(\bm x_0,\ldots,\bm x_d)d\bm x_0\ldots \bm x_d
	& =
	d!\int\limits_{\R^d}\int\limits_{0}^{\infty}\int\limits_{(S^{d-1})^{d+1}} f(\bm c +r \bm u_{0},\ldots,\bm c + r \bm u_{d})
	\notag \\ & \hskip2cm \times
	r^{d^2-1}\l_{d}([\bm u_{0},\ldots, \bm u_{d}]) ~d\bm u_{0}\ldots d \bm u_{d} dr d\bm c .
\end{align}

If we integrate only over $d$ points, there is also an integral transformation of Blaschke-Petkantschin type, since $d$ points together with the origin define a unique sphere. This can be found in \cite[Theorem 3]{Nick}. Let ${f:(\R^d)^{d}\to\R}$ be a nonnegative measurable function. Then
\begin{align}
	\label{eq:Bla1}
	\int\limits_{(\R^d)^{d}}f(\bm x_1,\ldots,\bm x_d)d\bm x_1\ldots d \bm x_d
	& =
	d!\int\limits_{0}^{\infty}\int\limits_{S^{d-1}}\int\limits_{(S^{d-1})^{d}} f(r \bm u+r \bm u_{1},\ldots,r\bm u+r \bm u_{d})
	\notag\\& \hskip1cm \times
	r^{d^2-1}\l_{d}([\bm u_{1},\ldots,\bm u_{d},-\bm u])~d\bm u_{1}\ldots d\bm u_{d} d\bm u dr.
\end{align}
Here $r \bm u$ is the center of the sphere and $\bm x_i= r \bm u+r \bm u_{i}$.

A Blaschke-Petkantschin integral transformation often leads to an integral of the form below. According to \cite[Theorem 8.2.3]{StoGeo}, this corresponds to a dimension-dependent  constant. For $d,k\in\N$ with ${d\ge 1}$ and ${k\ge0}$ it is
	\begin{align}
		\label{eq:S}
		S(d,d,k)
		&:=
		\int\limits_{(S^{d-1})^{d+1}}\l_{d}([\bm u_0,\ldots,\bm u_d])^k \, d\bm u_0\ldots d\bm u_{d}
		=
		\frac{\omega_{d+k}^{d+1} \kappa_{d(d+k-2)+d-2}}{(d!)^k \kappa_{(d+1)(d+k-2)} c_{(d+k)d}},
	\end{align}
where $c_{(d+k)d}$ is defined as $\frac{\omega_{k+1}\cdot\ldots\cdot\omega_{d+k}}{\omega_1\cdot\ldots\cdot\omega_{d}}$. Some basic calculations show that $$S(d,d,2)=\frac{d+1}{(d-1)!}\k_d^{d+1}.$$

\subsection{Poisson functionals}
Let $\eta_t$ be a stationary Poisson point process in $\R^d$ with intensity measure $\mu_t = t\l_d $ with $t > 0$, where $\l_d $ is the $d$-dimensional Lebesgue measure. Thus $\eta_t: \Omega \to \bm N(\R^d)$ is a random variable taking values in the space of $\s$-finite counting measures $\bm N(\R^d)$, $\eta_t (\omega, \cdot)= \sum \d_{\bm x_i(\omega)}(\cdot)$ where $\d_{\bm x}(\cdot)$ is the Dirac measure evaluated at the point $\bm x_i$. 
In the following we suppress the dependence on $\omega$. The intensity measure is the expectation
$\E \eta_t (A) = \E \sum \d_{\bm x_i} (A)=t \l_d(A)$.
We identify $\eta_t$ with its support $\{\bm x_i\}$, thus $\eta_t(A)$ is a Poisson random variable, the number of points in $A$, and $\eta_t \cap A$ denotes the set of points. 
For more information see for example \cite{StoGeo}\cite{LePoPu}.

A \emph{Poisson functional} $F$ is a random variable depending on a Poisson point process $\eta_t$. Let  $f: \bm N(\R^d) \to\R$ be a measurable function. If $F=f(\eta_t)$ almost surely, we call $f$ a representative of $F$. For simplicity we denote both, $F$ and $f$, with the symbol $F$. If $F$ is square-integrable, i.e. $\E \vert F\vert^2< \infty$, we write $F\in L_{\eta_t}^2$.

One main tool in this paper is the multivariate  Mecke equation, see for example \cite[Theorem 4.4]{LePoPu}. This equation provides an elegant way to compute the expectation. Let $f:\R^m\times \bm N(\R^d)\to \R$ be a measurable function, $ m \geq 1$. Then
\begin{equation}
	\label{eq:Mecke}
	\E\sum_{(\bm x_1,\ldots,\bm x_m)\in\eta^m_{\neq}}  f(\bm x_1,\ldots,\bm x_m; \eta_t)
	= 
	t^m\int\limits_{\R^m} \E f \left(\bm x_1,\ldots,\bm x_m;  \eta_t + \sum_1^m \d_{\bm x_i} \right)\, 
	d\bm x_1 \ldots  d\bm x_m .
\end{equation}

\subsubsection{Poisson Voronoi- and Delaunay-tessellations}

Let $\eta_{t, \neq}^m$ denote the set of all pairwise distinct $m$-tuples of points of a stationary Poisson point process $\eta_t$. Write $[\bm x_0, \dots, \bm x_d]$ for the convex hull of $(\bm x_0, \dots, \bm x_d) \in \eta_{t,\neq}^{d+1}$, and let $B(\bm x_0,...,\bm x_d)$ be the ball containing $\bm x_0,...,\bm x_d$ in its boundary, i.e. the circumball. The circumball exists and is unique with probability one.
A $(d+1)$-tuple $(\bm x_0, \dots, \bm x_d) \in \eta_{t,\neq}^{d+1}$ forms a Delaunay cell $[\bm x_0, \dots, \bm x_d]$, if the circumball $B (\bm x_0,...,\bm x_d)$ does not contain any further point from $\eta_t$. It is well known, see e.g. \cite{StoGeo}, that the collection $D_{\eta_t}$ of all Delaunay cells is a random tessellation of $\R^d$ called the Poisson-Delaunay tessellation. Define the center $\bm c=\bm c( \bm x_0,...,\bm x_d)$ of a Delaunay cell as the center of the circumball
$$B (\bm x_0,...,\bm x_d)=B(\bm c, \|\bm c -\bm x_0\|) , $$
and write $D_{\eta_t}(\bm c)$ for the corresponding Delaunay cell. Let $C(D_{\eta_t})$ be the set of all centers of Delaunay cells of $\eta_t$. As already stated in \eqref{def:PoisDel}, the Poisson-Delaunay approximation $A_{\eta_{t}}$ of $A$ is
$$
A_{\eta_{t}} :=\bigcup_{\bm c  \in C(D_{\eta_t}) \cap A} D_{\eta_t}(\bm c) .
$$

Let $V_{\eta_t}$ be the Voronoi tessellation of $\eta_t$, and define for general $\bm x \in \R^d$ by $V_{\eta_t + \d_\bm x} (\bm x)$ the Voronoi cell as the set of all points, which are closer to $\bm x$ than any other point of the  Poisson point process $\eta_t$, i.e. the cell with center $\bm x$ of the Voronoi tessellation of $\eta_t +\d_{\bm x}$. In the following, the radius $r_{\bm x}$ of the circumball of the Voronoi cell $V_{\eta_t + \d_\bm x} (\bm x)$ with center $\bm x$ will turn out to be of importance.

\begin{lemma}\label{le:rx}
Let $r_{\bm x}$ be the circumradius of $V_{\eta_t + \d_\bm x} (\bm x)$ with $t \geq 1$, and let $s \in \R$. Then
$$
\E r_{\bm x}^k \id(\bm r_{\bm x} \geq s)
\leq
S(d,d,1)\,  t^d
\int\limits_{0}^{\infty}
e^{-t\kappa_{d} r^{d}} r^{d^2+k-1} \id(r \geq s) \, dr
.
$$
\end{lemma}
\begin{proof} We start with the estimate
\begin{align*}
\E r_{\bm x}^k \id(\bm r_{\bm x} \geq s)
\leq
\E \sum_{\bm c \in \cF_0(V_{\eta_t + \d_\bm x} (\bm x)) } \| \bm x - \bm c  \|^k \id(\| \bm x - \bm c  \| \geq s ),
\end{align*}
where $\cF_0(V_{\eta_t + \d_\bm x} (\bm x)) $ denotes the vertices of the Voronoi cell $V_{\eta_t + \d_\bm x} (\bm x)$. Each vertex $\bm c$ is the center $\bm c(\bm x, \bm x_1, \dots, \bm x_d)$ of a circumball of a $d$-tuple in $\eta_{t, \neq}^d$ and $\bm x$, which is empty with respect to $\eta_t$. Using this and applying the multivariate Mecke formula \eqref{eq:Mecke} results in the bounds
\begin{align*}
\E r_{\bm x}^k \id(\bm r_{\bm x} \geq s)
&\leq
\frac{1}{d!}
\E \sum_{(\bm x_{1},\ldots,\bm x_{d})\in\eta_{t\neq}^{d}}
\id \left(\eta_{t} ( B^{\circ}(\bm x,\bm x_{1},\ldots, \bm x_{d}))= 0 \right)  \| \bm x - \bm c (\bm x,\bm x_{1},\ldots,\bm x_{d}) \|^k
\\ & \hskip5cm \times
  \id(\| \bm x - \bm c (\bm x,\bm x_{1},\ldots,\bm x_{d})  \| \geq s )
\\ &\leq
\frac{1}{d!} t^d \int\limits_{(\R^{d})^{d}}
\P(\eta_{t} ( B^{\circ}(\bm x_{0},\ldots, \bm x_{d}))=0)
\| \bm x - \bm c (\bm x,\bm x_{1},\ldots,\bm x_{d})\|^k
\\ & \hskip5cm \times
\id(\| \bm x - \bm c (\bm x,\bm x_{1},\ldots,\bm x_{d})  \| \geq s )\, d\bm x_1\ldots d \bm x_d
. 
\end{align*}
Because the number of points is Poisson distributed, we have 
$$
\P(\eta_{t} ( B^{\circ}(\bm x_{0},\ldots, \bm x_{d}))=0)=e^{-t\l_d(B(\bm x_{0},\ldots, \bm x_{d}))}
$$
which results in
\begin{align*}
	\E r_{\bm x}^k \id(\bm r_{\bm x} \geq s)
	&\leq
	\frac{1}{d!} t^d \int\limits_{(\R^{d})^{d}}
	e^{-t\k_d	\| \bm x - \bm c (\bm x,\bm x_{1},\ldots,\bm x_{d})\|^d}
	\| \bm x - \bm c (\bm x,\bm x_{1},\ldots,\bm x_{d})\|^k
	\\ & \hskip5cm \times
	\id(\| \bm x - \bm c (\bm x,\bm x_{1},\ldots,\bm x_{d})  \| \geq s )\, d\bm x_1\ldots d \bm x_d
	. 
\end{align*}
Since Lebesgue measure is translation invariant, this is independent of $\bm x$.
\begin{align*}
\E r_{\bm x}^k \id(\bm r_{\bm x} \geq s)
&\leq
\frac{1}{d!}  t^d \int\limits_{(\R^{d})^{d}}
e^{-t\kappa_{d} \| \bm c (\bm 0,\bm x_{1},\ldots,\bm x_{d}) \| ^{d}}
\| \bm c (\bm 0,\bm x_{1},\ldots,\bm x_{d})\|^k
\\ & \hskip5cm \times
\id(\| \bm c (\bm 0,\bm x_{1},\ldots,\bm x_{d})  \| \geq s )\, d\bm x_1\ldots d \bm x_d
\end{align*}
The Blaschke-Petkantschin integral transformation \eqref{eq:Bla1} with $\bm c (\bm 0,\bm x_{1},\ldots,\bm x_{d})  =r\bm u$ yields
\begin{align*}
\E r_{\bm x}^k \id(\bm r_{\bm x} \geq s)
&\leq
 t^d \int\limits_{0}^{\infty}
e^{-t\kappa_{d} r ^{d}} r^{d^2+k-1} \id( r \geq s )\, dr
\\ & \hskip4cm \times
\int\limits_{(S^{d-1})^{d+1}}
\l_{d}([\bm u_{1},\ldots,\bm u_{d},-\bm u])~d\bm u_{1}\ldots d\bm u_{d} d\bm u
\\ &\leq
S(d,d,1)\,  t^d
\int\limits_{0}^{\infty}
e^{-t\kappa_{d} r^{d}} r^{d^2+k-1} \id(r \geq s) \, dr
\end{align*}
because the second integral equals $S(d,d,1)$, see \eqref{eq:S}.
\end{proof}

\subsubsection{Variance bounds}

For $\bm x\in\R^d$, the first order difference of a Poisson functional $F$ is defined by
$$
D_{\bm x} F= F(\eta_t + \d_{\bm x})- F(\eta_t), 
$$
and for $n\in \N$ and $\bm x_1,...,\bm x_n\in\R^d$ we can define the differential operator of $n$-th order iteratively
$$
D_{\bm x_1,...,\bm x_n} ^n F= D_{\bm x_n}(D_{\bm x_1,...,\bm x_{n-1}}^{n-1} F).
$$
Now we can put the variance of a Poisson functional into terms of the difference operators, see for example \cite[Theorem 18.6]{LePoPu} or \cite[Theorem 1.1]{Fock}. Assume that $\eta_t$ has intensity measure $t \l_d$. The so called \emph{Fock space representation} of a Poisson functional $F$, with $F\in L_{\eta_t}^2$, is defined by
$$
\V(F)=\sum_{n=1}^{\infty}\frac{t^n}{n!}\int\limits_{\R^d}(\E D_{\bm x_1,...,\bm x_n} ^n F)^2 \, d\bm x_1,...d\bm x_n.
$$ 

From this representation one can derive upper an lower bounds for the variance. For $F\in L_{\eta_t}^2$, the Poincar\'e inequality for a Poisson point process, see for example \cite[Theorem 18.7]{LePoPu}, says that  
\begin{equation}
	\label{eq:poin}
	\V (F)\le t\int\limits_{\R^d}\E (D_{\bm x} F)^{2} d\bm x.
\end{equation}

For a lower variance bound we just consider the first summand of the Fock space representation. Thus we obtain
$$
\V(F)\ge t\int\limits_{\R^d}(\E D_{\bm x} F)^2 d\bm x.
$$
Unfortunately, this formula is not suitable to calculate a lower variance bound for all Poisson functionals. It becomes difficult, when the difference operator of a functional can be both, positive and negative. In this case, the positive and negative values can cancel out each other and the expectation of the sum can equal zero. As we will see later, this is also the case for the Poisson-Delaunay Approximation. To work around this problem, we will use a theorem by Schulte and Trapp \cite[Theorem 1.1]{SchuTra}, which provides a formula containing only the squared difference operator. They call it a counterpart of the Poincar\'e inequality. 

\begin{theorem}
	\label{thm:TraSchu}
	If a Poisson functional $F\in L_{\eta_t}^2$ satisfies 
	\begin{equation}
		\label{eq:Bed}
		t^2\int \E (D_{\bm x,\bm y }^2 F )^2 d\bm x d\bm y 
		\le
		\alpha t\int \E  (D_{\bm x} F )^2 d\bm x 
		<\infty
	\end{equation}
	for some constant $\alpha \ge 0$, then
	\begin{equation}
		\label{eq:V}
		\V F \ge \frac{4}{ (\alpha+2 )^2} \, t\int \E (D_{\bm x} F )^2 d\bm x .
	\end{equation}
\end{theorem}

\subsubsection{Normal approximation}
We need an upper bound for the Kolmogorov distance between a Poisson functional and a standard Gaussian random variable $N$, where the  Kolmogorov distance of two random variables $X,Y$ is defined by
\begin{equation*}
	\label{Kol}
	d_K(X,Y)=\sup_{x\in\R}\left\vert \P\left(X\le x\right)-\P\left(Y\le x\right)\right\vert.
\end{equation*}
Since convergence in distribution follows from convergence in Kolmogorov distance, this is a helpful tool to derive a central limit theorem. 
To bound the Kolmogorov distance we will use the following theorem, which can be found in \cite[Theorem 6.1]{LaPeSchu}. 
\begin{theorem} \label{thm:clt}
Let $F\in L_{\eta_t}^2$ with $\V F>0$ and denote by $N$ a standard Gaussian random variable. Assume that there are constants $c \geq 1$ and $p_1 , p_2 > 0$ such that for all $\bm x,\bm y \in \R^d$
\begin{equation*}
	\E \left\vert D_\bm x F \right\vert^{4+p_1}\le c
\text{ and }
	\E \left\vert D_{\bm x,\bm y }^2 F \right\vert^{4+p_2}\le c .
\end{equation*}
Then
\begin{align} \label{eq:dk}
d_K\left(\frac{F-\E F}{\sqrt{\V F}},N\right)
\le&
5c \left[
\frac{\Gamma_1^{1/2}}{\V F} +
\frac{ \Gamma_2^{1/2}}{\V F}+\frac{ \Gamma_2}{\left(\V F\right)^{3/2}}+\frac{ \Gamma_2^{5/4}+\Gamma_2^{3/2}}{\left(\V F\right)^2}
+ \frac{\Gamma_3^{1/2}}{\V F}
\right]
\end{align}
with
\begin{align*}
\Gamma_1 & =t\int\left(t\int\P\left(D_{\bm x,\bm y }^2F\neq 0\right)^{p_2/(16+4p_2)}d\bm y \right)^2d\bm x ,
\\
\Gamma_2 & =t\int\P\left(D_{\bm x}F\neq 0\right)^{p_1/(8+2p_1)}d\bm x ,
\\
\Gamma_3 & = t^2\int \P\left(D_{\bm x,\bm y }^2F\neq0\right)^{p_2/(8+2p_2)}d\bm x d\bm y
.
\end{align*}
\end{theorem}

\section{Proof of Theorem \ref{thm:exp}}
\label{Esti}
Due to the definition of the Poisson-Delaunay approximation, 
\begin{align*}
	\l_d(A_{\eta_t}) 
	& =
	\sum_{\bm c \in C(D_{\eta_{t}})\cap A} \l_d(D_{\eta_{t}}(\bm c)) 
	\\ & =
	\frac{1}{(d+1)!} \sum_{(\bm x_{0},\ldots,\bm x_{d})\in\eta_{t,\neq}^{d+1}} \id(\eta_{t} (B^\circ (\bm x_{0},\ldots, \bm x_{d}))=0) 
	\\& \hskip4cm \times
	\id(\bm c  (\bm x_0,...,\bm x_d)\in A) \l_d([\bm x_0,...,\bm x_d]). 
\end{align*}
The multivariate Mecke equation \eqref{eq:Mecke} yields
\begin{align*}
	\E \l_d(A_{\eta_t})  
	& =
	t^{d+1}\frac{1}{(d+1)!} \int\limits_{(\R^d)^{d+1}}e^{-t\l_d(B(\bm x_{0},\ldots, \bm x_{d}))}\id(\bm c( \bm x_0,...,\bm x_d)\in A)
	\\& \hskip6.5cm \times
	 \l_d([\bm x_0,...,\bm x_d]) ~d\bm x_0...d\bm x_d.
\end{align*}
Combined with the Blaschke-Petkantschin transformation \eqref{eq:Bla} and \eqref{eq:S}, this leads to
\begin{align*}
	&\E ~\l_d(A_{\eta_t})  
	\\&=
	t^{d+1}\frac{1}{d+1} \int\limits_{\R^d}\int\limits_0^{\infty}\int\limits_{(S^{d-1})^{d+1}}e^{-t\kappa_d r^d}\id(\bm c \in A) r^{d^2+d-1}\l_d([\bm u_0,...,\bm u_d])^2 ~d\bm u_0...d\bm u_d dr d\bm c 
	\\ &=
	t^{d+1}\frac{1}{d+1}\l_d(A) S(d,d,2) \int\limits_0^{\infty}e^{-t\kappa_d r^d} r^{d^2+d-1}~dr
	\\ &=
	\frac{1}{d(d+1) \kappa_d^{d+1}}\l_d(A) S(d,d,2) \Gamma(d+1).
\end{align*}
Since $S(d,d,2)=\frac{d+1}{(d-1)!}\kappa_d^{d+1}$, we finally obtain $\E \l_d(A_{\eta_t})  =\l_d(A)$.

\section{Proof of Theorem \ref{thm:var}}

For the upper bound of the variance we use the Poincar\'e inequality \eqref{eq:poin}. First, we have to bound the difference operator of the Poisson-Delaunay approximation from above. Adding a point $\bm x$ changes all Delaunay cells, for which $\bm x$ lies inside of the circumball. In this area new cells with $\bm x$ as a vertex are formed. In particular, the range that is affected by adding the point $\bm x$ consists of these new cells. To quantify the affected range, we consider the dual Poisson-Voronoi tessellation $V_{\eta_t}$. The additional point $\bm x$ is the center of a new Voronoi cell $V_{\eta_t+\delta_{\bm x}}(\bm x)$. By duality, each facet of the new Voronoi cell is halving one edge of the new Delaunay cells which contain $\bm x$. Thus, all vertices of the new Delaunay cells, and hence all the cells itself, are contained in twice the new Voronoi cell $2 V_{\eta_t+\d_\bm x}(\bm x)$. Since the additional point $\bm x$ changes all Delaunay cells in $D_{\eta_t}$ for which $\bm x$ lies inside of the circumball, the centers of these old cells are closer to $\bm x$ than any other point in $\eta_t$. For this reason, the centers of the involved old cells lie inside of the new Voronoi cell $V_{\eta_t+\d_\bm x}(\bm x)$. Further, again by duality, the centers of the new Delaunay cells are the vertices of the new Voronoi cell and thus are also contained in $2 V_{\eta_t+\d_\bm x}(\bm x)$.

Remind that $r_{\bm x}$ is the radius of the circumball with center $\bm x$ containing  $V_{\eta_t+\delta_{\bm x}}(\bm x)$. We just have shown, that the difference between $\l_d(A_{\eta_t}) $ und $\l_d(A_{\eta_t+\delta_{\bm x}})$ is bounded by the volume of the ball $ B (\bm x,2r_{\bm x})$ which also contains all centers of the involved old and new Delaunay cells.

If $ B (\bm x,2r_{\bm x})$, and thus the centers of all involved Delaunay cells, is completely inside or outside of the set $A$, the Poisson-Delaunay approximation $A_{\eta_t}$ remains unchanged. Let $\partial A$ be the boundary of the set $A$. Thus, only if $\partial A$ intersects the ball $ B (\bm x,2 r_{\bm x})$ the difference operator can be different to zero. In total, we obtain
\begin{equation}\label{eq:Dx<}
\vert D_{\bm x}\l_d (A_{\eta_{t}} )\vert
\le \kappa_{d}2^dr_{\bm x}^{d} \id ( \bm x \in \partial A+2 r_{\bm x} B_d )
= \kappa_{d}2^dr_{\bm x}^{d} \id ( r_{\bm x} \geq {\textstyle \frac 12} d(\bm x , \partial A )).
\end{equation}
To make use of the Poincar\'e inequality \eqref{eq:poin} we need the squared difference operator. Thus, we have to estimate $\E r_{\bm x}^{2d} $. By Lemma \ref{le:rx} and integrating with respect to $\bm x$ we obtain
\begin{align*}
\V \l_d (A_{\eta_{t}} )
& \le
\kappa_{d}^{2} 2^{2d} S(d,d,1) t^{d+1}\int\limits_{0}^{\infty}e^{-t\kappa_dr^d}r^{d^2+2d-1}\l_d (\partial A+ 2r B_d) dr
.
\end{align*}
Because $A$ is assumed to be convex, we estimate the volume of ${\partial A+ 2 r B_d}$ using the Steiner formula \eqref{eq:stein},
\begin{equation}\label{eq:stein-estimate}
\l_d \left( \partial A + 2r B_d\right)\le2\left(\l_d (A+ 2r B_d)-\l_d (A)\right)=2\sum_{i=1}^{d} 2^i \kappa_iV_{d-i}(A)r^i.
\end{equation}
In total, we receive
\begin{align*}
\V \l_d(A_{\eta_{t}})
& \le
\sum_{i=1}^d \kappa_{d}^{2} 2^{2d+i+1}t^{d+1} S(d,d,1)\kappa_i V_{d-i}(A) \int\limits_{0}^{\infty}e^{-t\kappa_dr^d}r^{d^2+2d-1+i}dr 
\\ & =
\sum_{i=1}^d 2^{2d+i+1}  d^{-1} \kappa_{d}^{-d -\frac{i}{d}} \kappa_i    t^{-1- \frac{i}{d}} S(d,d,1)  V_{d-i}(A)  \Gamma\left(\frac{d^2+2d+i}{d}\right) . 
\end{align*}
Finally, we summarize all factors depending on $i$, $d$ and $A$ to a constant $c_{d,A}$ and obtain
\begin{equation*} \label{eq:lang}
\V \left(A_{\eta_{t}}\right) \le
c_{d,A}~t^{-1-\frac{1}{d}}
\end{equation*}
for $t$ bounded away from zero, e.g. $t \geq ({8d}/{r_A} )^d $.
This is the upper bound for the variance of $\l_d(A_{\eta_t}) $ in Theorem \ref{thm:var}.

\medskip
Next, we will proof the existence of a variance lower bound of the same order. As mentioned, it is not easy to estimate the difference operator from below, because it attains positive and negative values. For the Voronoi approximation, this problem was easy to solve, since $D_\bm x (V_{\eta_t,A })\ge 0$ for $\bm x \in A$ and $D_\bm x (V_{\eta_t,A })\le 0$ for $\bm x \notin A$, see \cite{Schulte}. This is not true in the case of the Delaunay approximation, and thus a more complex construction is required. We apply the theorem by Trapp and Schulte, Theorem \ref{thm:TraSchu}, and thus have to show the existence of an $\alpha$ such that the condition \eqref{eq:Bed} from Theorem \ref{thm:TraSchu} is fulfilled. For this purpose we first have to determine a lower bound for $ \E ( D_{\bm x} \l_d(A_{\eta_t}) )^2 $.

For $\bm y\in\partial A$, let $r_A(\bm y)$ be the largest ball, which contains $\bm y$ and is contained in $A$. 
By Lemma 4 in \cite{floating} (see also \cite[Lemma 1.5.2]{DissHug}) we have
\begin{equation}
\label{eq:Haus}
\mathcal{H}^{d-1}\left(\{\bm y \in\partial A \colon r_A(\bm y)\ge\e \}\right)
\ge
\left(1-\frac{\e }{r_A}\right)^{d-1} \kappa_1 V_{d-1}(A)
\end{equation}
for $\e \leq r_A$.
Let $\text{proj}_{\partial A}(\bm x)$ be the metric projection of a point $\bm x\in\R^d$ on $\partial A$. We consider the set
\begin{equation*}
M_\e :=\left\{ \bm x\in A^c \cap (\partial A + \e B_d )\colon r_A\left(\text{proj}_{\partial A}(\bm x)\right)\ge 4d\e  \right\}
\end{equation*}
for $\e \le \frac{r_A}{8d}$. This is the set of all points outside $A$ but $\e$-close to its boundary, which have a $4d\e$-ball inside $A$.
The volume of the set $M_\e $ can be bounded from below by \ref{eq:Haus} and the condition $\e \le \frac{r_A}{8d}$. Since $A$ is convex we obtain
\begin{align}
\l_d \left(M_{\e }\right)
&\ge 
\mathcal{H}^{d-1}\left(\left\{\bm y \in\partial A \colon r_A(\bm y)\ge 4d \e \right\}\right)\e   
\notag \\&\ge 
\left(1-\frac{4d \e }{r_A}\right)^{d-1} \kappa_1 V_{d-1}(A) \e
\notag \\ &\ge
2^{-d+2} V_{d-1}(A) \e.
\label{eq:Vol}
\end{align}
We consider the difference operator of the Poisson-Delaunay approximation, but now only for points $\bm x$ in $M_\e $. We will show that with positive probability an event $E$ occurs which guarantees that the squared difference operator is of order $t^{-2}$.

To this end we use the construction from Lemma \ref{le:simplex}. For $\bm x \in M_\e$ let $L_\bm x$ be the line containing $\bm x$ and $\text{proj}_A(\bm x)$.
Denote by $\bm c  \in L_\bm x$ the point in $A$ with distance $2 \e$ to $\text{proj}_A(\bm x)$. Define $T$ to be a a regular simplex with center $\bm c $ and inradius $5 \e$ (thus having circumradius $5 d \e$), chosen such that one vertex $\bm a _{d+1} \in A$ is on $L_\bm x$ with distance $5 d \e$ to the center $\bm c $, and thus the vertices $\bm a _1, \dots, \bm a _d$ are contained in a hyperplane $H \subset A^c$ orthogonal to $L_\bm x$ with $d(H\cap L_\bm x,\bm c )=5 \e $. This construction is chosen in such a way that $B (\bm c ,\e) \subset A$ and $\bm x \in B ( \bm c , 3 \e)$.


We apply Lemma \ref{le:simplex} with $s= 5\e$. We form a ball around each of the vertices $\bm a _1, ..., \bm a _{d+1}$ with radius $5\d \e  $ where $\d$ is chosen according to Lemma \ref{le:simplex}. Because $\d < \frac 15 $ the balls $ B (\bm a _i,5 \d \e  ) $ are disjoint.
We define the event $E$ by assuming that in each of the balls $ B (\bm a _i,5 \d \e  ) $  there is precisely one point $\bm x_i$ from $\eta_t$, i.e. $\eta_t \cap  B (\bm a _i, 5 \d  \e  ) = \{\bm x_i\}$ for $i\in\{1,..,d+1\}$, and further that all other points from $\eta_t$ are far away from $\bm c$,
\begin{equation*}
\eta_t \left( B (\bm c , 34 d^2  \e)\setminus \bigcup_{i\in\{1,...,d+1\}} B (\bm a _i,5 \d \e )\right)= 0.
\end{equation*}

This implies that the simplex $Z=[\bm x_1, \dots, \bm x_{d+1}]$ with $\bm x_i \in B(\bm a _i, 5 \d \e)$ is a Delaunay cell of $\eta_t$. Assume that $B$ is a ball containing $d$ points from the set of all vertices $\{\bm x_1,...,\bm x_{d+1}\}$ and one further point $\bm y \in\eta_t \setminus B(\bm c , 34 d^2  \e)$ on its boundary. Then such a ball $B$ has radius at least $ 15 d^2 \e$, since
$$
\| \bm y -\bm x_i\|\ge  34 d^2\e - \| \bm c -\bm x_i\| \geq 34 d^2\e - (5d\e + 5 \d\e) \geq  30 d^2\e
$$
for all $i\in\{1,...,d+1\}$. According to Lemma \ref{le:simplex} the ball $B$ is disjoint from $B(\bm c , 4 \e)$. In particular, it can not contain $\bm x$ which is in $B(\bm c , 3\e)$.
Under this conditions the Delaunay cell $Z$ is the only cell, which contains $\bm x$ in the circumball and thus the only cell which changes when adding the point $\bm x$ to $\eta_t$. Furthermore, the simplex $Z_\bm x=[\bm x_1, \dots, \bm x_{d}, \bm x] \subset Z$ is a Delaunay cell of $\eta_t+\delta_\bm x$ and the center of this cell is not contained in $A$.

Summarizing, by Lemma \ref{le:simplex} the Delaunay simplex $Z=[\bm x_1, \dots, \bm x_{d+1}]$ has its center $\bm c'$ in $B (\bm c ,\e)$, which is a subset of $A$, and therefore belongs to $A_{\eta_t}$. But the new Delaunay cell $Z_\bm x=[\bm x_1, \dots, \bm x_{d}, \bm x]$, which is a subset of $Z$, has its center in $A^c$ and thus does not belong to the Poisson-Delaunay approximation $A_{\eta_t+\delta_\bm x}$. Thus, given the event $E$ we have
$$
D_\bm x \l_d(A_{\eta_t})^2  \ge \l_d(Z_{\bm x})^2.
$$

The volume of the simplex $Z_\bm x$ obviously depends on $\e $ and the dimension $d$. Since the height of the simplex is bounded from below by $\e \cdot(2 - 5 \d)$, and the inradius of the base simplex $[\bm x_1,...,\bm x_d]$ is bounded by $\e \cdot(5 \sqrt{ \frac {d+1}{d-1} }-5 \d)$, there is a constant $\underline{c_d}>0$ such that
\begin{equation*}
\l_d (Z_\bm x)
\geq \underline{c_d}\e ^d ,
\end{equation*}
and choosing $\e = t^{- \frac 1d}$ gives
\begin{equation}
\label{eq:E}
\E \left(D_\bm x \l_d(A_{\eta_t})^2 |E\right)\ge \underline{c_d}^2 t^{-2}.
\end{equation}
Because $\e = t^{- \frac 1d}$, the probability that the conditions $E$ occur is
\begin{align}\label{eq:P}
\P 
&
\Big( \eta_t\big(B(\bm c ,34 d^2 \e )\setminus  \! \bigcup_{i\in\{1,...,d+1\}} \! B(\bm a _i,5\d \e )\big)=0,\ \eta_t\left(B(\bm a _i, 5\d \e )\right)=1,\, i=1,\dots,d+1  \Big)
\notag \\
& =
e^{-t\left(34 d^2 \e \right)^d \kappa_d + (d+1) t (5 \d \e)^d \kappa_d}
\prod_{i=1}^{d+1} t \left(5 \d \e \right)^d \kappa_d e^{-t\left(5 \d \e \right)^d\kappa_d}
= c_d' >0 .
\end{align}
Next, we apply the law of the total expectation. Together with \eqref{eq:E} and \eqref{eq:P} we obtain
\begin{align}\label{eq:upperboundTH3}
t \int \E \left(D_\bm x\left(\l_d(A_{\eta_t}) \right)^2\right)d \bm x
&\ge \nonumber
t \int \E \left(D_\bm x \l_d(A_{\eta_t})^2 | E\right)\cdot\P(E) \id(\bm x \in M_{t^{-\frac 1d}})\, d\bm x
\\ & \ge \nonumber
t \underline{c_d}^2 t^{-2} c_d' \l_d(M_{t^{-\frac 1d}})
\\ & \ge
 \underline{c_d}^2 c_d'
2^{-d+2} V_{d-1}(A) t^{-1- \frac 1d}
\end{align}
for $t \geq ({8d}/{r_A} )^d $. The last step follows from the volume estimation \eqref{eq:Vol} of $M_\e$.

\medskip
To compute $\alpha$ in condition \eqref{eq:Bed}, we also need an upper bound for the second order difference. The second order difference
$
D_{\bm x,\bm y }^2 \l_d(A_{\eta_t}) = D_{\bm y } D_{\bm x} \l_d(A_{\eta_t})
$
is equal to zero, when the distance between the points $\bm x$ and $\bm y $ is too large. Recall that adding $\bm x$ generates new Delaunay cells which have as centers the vertices of the new Voronoi cell generated by $\bm x$ and thus are contained in $B(\bm x, 2r_{\bm x})$, where again $r_\bm x$ is the radius of the circumball of the Voronoi cell $V_{\eta_t+\delta_x}(x)$. If $\bm y$ is not in this region then it cannot change the Delaunay cells generated by $\bm x$ and thus
\begin{align*}
|D_{\bm x,\bm y }^2 \l_d(A_{\eta_t})|
&=
|D_{\bm x,\bm y }^2 \l_d(A_{\eta_t})| \id( \bm y \in B(\bm x, 2r_{\bm x}))
\\ & \leq
\big( |D_{\bm x} \l_d(A_{\eta_t+ \d_\bm y})| + |D_{\bm x} \l_d(A_{\eta_t})| \big) \id( \bm y \in B(\bm x, 2r_{\bm x})).
\end{align*}
Next, we use \eqref{eq:Dx<} and note that $r_{\bm x} = r_{\bm x}(\eta_t) \geq r_{\bm x}(\eta_t+\d_{\bm y}) $ because adding a new point cannot increase the Voronoi cell around $\bm x$. Thus
\begin{align}\label{eq:Dxy<}
|D_{\bm x,\bm y }^2 \l_d(A_{\eta_t})|
& \leq
2 \kappa_d 2^{d} r_{\bm x}^d \id (\bm x \in \bd A + 2 r_{\bm x} B_d)  \id( \bm y \in B(\bm x, 2r_{\bm x}))
.
\end{align}
Note that integrating $ \id( \bm y \in B(\bm x, 2r_{\bm x})) $ with respect to $\bm y$, which occurs in the left hand side of \eqref{eq:Bed}, leads to a factor $\kappa_d 2^d r_{\bm x}^d$.
Again, rewriting $ \bm x \in \bd A + 2 r_{\bm x} B_d$ as $r_{\bm x} \leq {\textstyle \frac 12} d(\bm x, \bd A)$, using Lemma \ref{le:rx} and integrating with respect to $\bm x$ we obtain
\begin{align*}
t^2 \int\limits_{(\R^{d})^2} \E \left( D_{\bm x,\bm y }^2 \l_d (A_{\eta_t})\right)^2 d\bm x d\bm y
& \le
\kappa_d^2 2^{3d+2} S(d,d,1)\,  t^{d+2}
\\ & \hskip1cm \times
\int\limits_{0}^{\infty}
e^{-t\kappa_{d} r^{d}} r^{d^2+3d-1} \l_d (\bd A + 2 r_{\bm x} B_d)  \, dr
.
\end{align*}
The Steiner formula \eqref{eq:stein}, in particular the estimate \eqref{eq:stein-estimate} gives
\begin{equation}\label{eq:lowerboundTH3}
t^2 \int\limits_{(\R^{d})^2} \E \left( D_{\bm x,\bm y }^2 \l_d (A_{\eta_t})\right)^2 d\bm x d\bm y
\le \tilde{c}_{d,A}t^{-1-\frac{1}{d}}.
\end{equation}
with a constant $\tilde{c}_{d,A}$ depending on the dimension $d$ and the convex set $A$.

\medskip
Now all preparations to apply Theorem \ref{thm:TraSchu} are done. By \eqref{eq:upperboundTH3} and \eqref{eq:lowerboundTH3}, the condition \eqref{eq:Bed} is fulfilled for
$$
\alpha=\frac{\tilde{c}_{d,A}}{\underline{c_d}^2 c_d'
2^{-d+2} V_{d-1}(A)}
$$
for $t$ bounded away from zero, e.g. $t \geq ({8d}/{r_A} )^d $.
Inserting our results into \eqref{eq:V} finally yields the lower bound for the variance of $\l_d(A_{\eta_t})$ in Theorem \ref{thm:var}.

\section{Proof of Theorem \ref{thm:CLT}}

In order to apply Theorem \ref{thm:clt} we consider the suitable normalized Poisson functional
$t \l_d(A_{\eta_t}) $. First, we put $p_1=p_2=1$  and check the conditions that
$\E \vert D_{\bm x} ( t \l_d(A_{\eta_t})) \vert^5 $ and $\E \vert D_{\bm x, \bm y}^2( t \l_d(A_{\eta_t}))\vert^5$
are finite. Because of \eqref{eq:Dx<} and \eqref{eq:Dxy<} and using Lemma \ref{le:rx} it suffices to show, that
$$
t^5\E r_{\bm x}^{5d}
\leq
S(d,d,1)\,  t^{d+5} \int\limits_{0}^{\infty} e^{-t\kappa_{d} r^{d}} r^{d^2+5d-1} \, dr
$$
is bounded, which is immediate.

The lower variance bound for the Poisson-Delaunay approximation $\l_d(A_{\eta_t}) $ in Theorem \ref{thm:var} gives
$$
\V ( t \l_d(A_{\eta_t}))
= t^2 \V \l_d(A_{\eta_t}) \ge t^{1-\frac{1}{d}} \underline{c_{d,A}}.
$$
In particular, $\V ( t \l_d(A_{\eta_t}))> 0$ for $t\ge (\frac{8d}{r_A})^d$. 

Thus, we can apply Theorem \ref{thm:clt}.
It remains to estimate $\Gamma_1, \Gamma_2$ and $\Gamma_3$.
To start, we use Lemma \ref{le:rx} with $k=0$,
\begin{align*}
\P (r_{\bm x} \geq s )^{\frac 1{m}}
& \le
S(d,d,1)^{\frac 1{m}} \,  t^{\frac d{m}}  \left(
\int\limits_{s}^\infty
e^{-t\kappa_{d} r^{d}} r^{d^2-1} \, dr
\right)^{\frac 1{m}} .
\end{align*}
Elementary calculations show that 
$ e^{-x} x^\b \leq \b^\b e^{-\frac{e-1}{e} x} \leq \b^\b e^{-\frac{1}{2} x} $. Therefore, we have
$$
e^{-t\kappa_{d} r^{d}} r^{d^2-d}
\leq
(d-1)^{d-1}  \k_d^{-(d-1)}  t^{-(d-1)} e^{-\frac 12 t\kappa_{d} r^{d}}
$$
and hence
\begin{align}\label{eq:PDneq0}
\left(\P (r_{\bm x} \geq s) \right)^{\frac 1{m}}
& \le \nonumber
S(d,d,1)^{\frac 1{m}} (d-1)^{\frac{d-1}{m}} \k_d^{-\frac{d}{m}}
\left(
\int\limits_{s}^\infty
e^{-\frac 12 t\kappa_{d} r^{d}}
t \k_d r^{d-1} \, dr
\right)^{\frac 1{m}}
\\ & =  \nonumber
S(d,d,1)^{\frac 1{m}} (d-1)^{\frac{d-1}{m}} 2^{\frac 1{m}} d^{-\frac 1{m}}   \k_d^{-\frac{d}{m}}
e^{-\frac 1{2m} t\kappa_{d} s^{d}}
\\ & =
S(d,d,1)^{\frac 1{m}} (d-1)^{\frac{d-1}{m}}2^{\frac 1{m}-1} m^{-1} d^{1-\frac 1{m}}  \k_d^{1-\frac{d}{m}}\, t
\int\limits_{s}^\infty
e^{-\frac 1{2m} t\kappa_{d} r^{d}}
r^{d-1}
dr
\end{align}
where in the last line we have written the exponential function again as an integral.

First, we insert this expression into the definition of
$\Gamma_2$,
$$
\Gamma_2 = t\int\limits_{\R^d} \P\left(D_{\bm x} ( t \l_d(A_{\eta_t}))  \neq 0\right)^{1/10}d\bm x
. $$
Because \eqref{eq:Dx<} implies
$$
\id( D_{\bm x} ( t \l_d(A_{\eta_t})) \neq 0)
\leq
\id(r_{\bm x} \geq \frac 12 d(\bm x, \bd A)) ,
$$
we put $m=10$ and $s= \frac 12 d(\bm x, \bd A)$ in \eqref{eq:PDneq0} and and obtain by Fubini's theorem
\begin{align*}
\Gamma_2
& \le
c_d \,  t^2
\int\limits_{0}^\infty
\int\limits_{\R^d}
e^{-\frac 1{20} t\kappa_{d} r^{d}}
r^{d-1}
\id(r \geq \frac 1{2} d(\bm x, \bd A))
d \bm x dr
\\ & \le
c_d \,  t^2
\int\limits_{0}^\infty
e^{-\frac 1{20} t\kappa_{d} r^{d}}
r^{d-1}
\l_d( \bd A + 2r B_d)
dr.
\end{align*}
Estimating the parallel volume as in \eqref{eq:stein-estimate} yields
\begin{align*}
\Gamma_2
& \le
\sum_{i=1}^d c_{d,A,i} \,  t^2
\int\limits_{0}^\infty
e^{-\frac 1{20} t\kappa_{d} r^{d}} r^{d+i-1}
dr
\le
c_{d,A} t^{1- \frac 1d} .
\end{align*}
Thus, $\Gamma_2$ is of order $t^{1-\frac{1}{d}}$.

Next, we turn our attention to the second order differences. Inequality \eqref{eq:Dxy<} implies
\begin{align*}
\id(D_{\bm x,\bm y }^2 (t \l_d(A_{\eta_t})) \neq 0)
& \leq
\id (\bm x \in \bd A + 2 r_{\bm x} B_d)  \id( \bm y \in B(\bm x, 2r_{\bm x}))
\\ & =
\id (r_{\bm x} \geq \frac 12 \max\{ d(\bm x ,\bd A ),  \| \bm x - \bm y \|  \} )
.
\end{align*}
Put $s=\frac 12 \max\{ d(\bm x ,\bd A ),  \| \bm x - \bm y \|  \}$. For $\Gamma_3$ it follows as above from \eqref{eq:PDneq0} that
\begin{align*}
\Gamma_3
& =
t^2\int\limits_{(\R^d)^2} \P\left( D_{\bm x,\bm y }^2 (t \l_d(A_{\eta_t})) \neq 0 \right)^{1/10}d\bm x d\bm y
\\ & \leq
t^2\int\limits_{(\R^d)^2} \P\left( r_{\bm x} \geq s \right)^{1/10}d\bm x d\bm y
\\ & \leq
c_d t^3
\int\limits_{0}^\infty
\int\limits_{(\R^d)^2}
e^{-\frac 1{20} t\kappa_{d} r^{d}}
r^{d-1} \id(r \geq \frac 12 \max\{ d(\bm x ,\bd A ),  \| \bm x - \bm y \|  \})
d\bm x d\bm y
dr
\\ & =
c_d t^3
\int\limits_{0}^\infty
\int\limits_{(\R^d)^2}
e^{-\frac 1{20} t\kappa_{d} r^{d}}
r^{d-1} \id( \bm x \in \bd A + 2 r B_d,\  \bm y \in B(\bm x, 2r) )
d\bm x d\bm y
dr
.
\end{align*}
The integration with respect to $\bm y$ yields a factor $2^d r^d \kappa_d$, the integration with respect to $\bm x$ as before a factor $\l_d(\bd A + 2 r B_d)$. The estimate \eqref{eq:stein-estimate} for the parallel volume leads to
\begin{align*}
\Gamma_3
 & \leq
\sum_{i=1}^d c_{d,A,i} t^3
\int\limits_{0}^\infty
e^{-\frac 1{20} t \kappa_{d} r^{d}}
r^{2d+i-1} dr
\leq
c_{d,A} t^{1-\frac 1d }
,
\end{align*}
and thus also $\Gamma_3$ is of order $t^{1-\frac{1}{d}}$.

Finally, for $\Gamma_1$ the same considerations with $s=\frac 12 \max\{ d(\bm x ,\bd A ),  \| \bm x - \bm y \|  \}$ but now with $m=20$ lead to
\begin{align*}
\Gamma_1
& =
t^3 \int\limits_{\R^d} \left( \int\limits_{\R^d} \P\left(D_{\bm x,\bm y }^2 (t \l_d(A_{\eta_t})) \neq 0\right)^{1/20}d\bm y \right)^2d\bm x
\\ & \leq
 t^3 \int\limits_{\R^d} \left( \int\limits_{\R^d} \P\left( r_{\bm x} \geq s \right)^{1/20}d\bm y \right)^2d\bm x
\\ & \leq
c_d t^5 \int\limits_{\R^d}
\left( \int\limits_0^\infty \int\limits_{\R^d}
e^{-\frac 1{40} t\kappa_{d} r^{d}}
r^{d-1}
\id( \bm x \in \bd A + 2 r B_d,\,  \bm y \in B(\bm x, 2r) )
dr
d\bm y \right)^2  d\bm x
.
\end{align*}
The integration with respect to $\bm y$ yields a factor $2^d r^d \kappa_d$. Then we write $\left(\int h(r) dr\right)^2$ as $\int\int h(r)h(s) dr ds$ and integrating with respect to $\bm x$ leads to a factor \linebreak ${\l_d(\bd A + 2 \min\{r,s\} B_d)}$,
\begin{align*}
\Gamma_1
& \leq
c_d t^5 \int\limits_{\R^d}
\left( \int\limits_0^\infty
e^{-\frac 1{40} t\kappa_{d} r^{d}}
r^{2d-1}
\id( \bm x \in \bd A + 2 r B_d)
dr
\right)^2 d\bm x
\\ & \leq
c_d t^5 \int\limits_{\R^d}
\int\limits_0^\infty
\int\limits_0^\infty
e^{-\frac 1{40} t\kappa_{d} (r^{d}+s^d)}
(rs)^{2d-1}
\id( \bm x \in \bd A + 2 \min\{r,s\} B_d)
dr ds \bm x
\\ & =
c_d t^5
\int\limits_0^\infty
\int\limits_0^\infty
e^{-\frac 1{40} t\kappa_{d} (r^{d}+s^d)}
(rs)^{2d-1}
\l_d( \bd A + 2 \min\{r,s\} B_d)
dr ds
.
\end{align*}
Using again the Steiner type estimate \eqref{eq:stein-estimate} gives
$$
\Gamma_1
\leq
\sum_{i=1}^d c_{d,A,i} t^{5}
\int\limits_0^\infty
\int\limits_0^\infty
e^{-\frac 1{40} t \kappa_{d} (r^{d}+s^d)}
(rs)^{2d-1} \min\{r,s\}^i dr ds
\leq c_{d,A} t^{1- \frac 1d}
.
$$
Thus, $\Gamma_1, \Gamma_2$ and $\Gamma_3$ are of order $t^{1-\frac{1}{d}}$.
Inserting all results in \eqref{eq:dk} finally leads to
\begin{align*}
d_K \left(\frac{\l_d(A_{\eta_t}) -\E \l_d(A_{\eta_t}) }{\sqrt{\V \l_d(A_{\eta_t}) }},N\right)
=
d_K\left(\frac{( t \l_d(A_{\eta_t}))-\E ( t \l_d(A_{\eta_t}))}{\sqrt{\V ( t \l_d(A_{\eta_t}))}},N\right)
\le
c_1 t^{-\frac{1}{2}+\frac{1}{2d}}.
\end{align*}

\textsl{Remark:} The abstract central limit theorem by Last, Pecatti and Schulte, see Theorem \ref{thm:clt}, also provides bounds for the Wasserstein distance. Since this bound has quite similar components compared to the bound for the Kolmogorov distance, the Wasserstein distance between the Poisson-Delaunay approximation and a Gaussian random variable has the same order. From both, convergence in Wasserstein distance and convergence in Kolmogorov distance, the existence of a central limit theorem can be concluded.

\section{Proof of Theorem \ref{thm:sym}}
\label{Sym}
Unlike before, $A$ can be any Borel set with bounded perimeter. We split the symmetric difference in the two parts outside and inside of the set $A$. Thus, we can rewrite the expected volume in the following way
\begin{equation*}
\E \l_d (A \Delta A_{\eta_t})=\E \l_d ( A_{\eta_t} \setminus A)+\E \l_d (A \setminus A_{\eta_t}).
\end{equation*}
Because $0=\E \l_d(A_{\eta_t}) - \l_d(A)  = \E \l_d ( A_{\eta_t} \setminus A) - \E \l_d (A \setminus A_{\eta_t})$,
we see that the parts inside and outside of $A$ have in mean the same size (although their probability distribution is different). Thus, $\E \l_d (A \Delta A_{\eta_t})=2\E \l_d ( A_{\eta_t}\setminus A)$, and it suffices to consider the first summand $\E \l_d ( A_{\eta_t}\setminus A)$.
\begin{align*}
\E \l_d (A \Delta A_{\eta_t})
& =
2 \E \int\limits_{\R^d} \id(\bm y  \in A^c,\ \bm y  \in A_{\eta_t}) d\bm y
\\ & =
2 \E \int\limits_{\R^d} \id(\bm y  \in A^c)
\frac{1}{(d+1)!}  \sum_{\left(\bm x_0,...,\bm x_{d}\right)\in\eta_{t_{\neq}}^  {d+1}} 
\id \left(\eta_t(B^{\circ}(\bm x_0, \dots, \bm x_{d})) =0\right) 
\\ & \hskip2cm \times
\id \left(\bm c (\bm x_0, \dots, \bm x_{d}) \in A\right)
\id (  \bm y  \in [\bm x_0, \dots, \bm x_{d} ]) d\bm y
\\ & =
\frac{2}{(d+1)!}  t^{d+1} \int\limits_{\R^d} \int\limits_{(\R^d)^{d+1}}
\id\left(\bm y  \in A^c\right) \id \left(\bm c (\bm x_0, \dots, \bm x_{d})  \in A\right)
\\ & \hskip2cm \times
e^{- t \l_d (B (\bm x_0, \dots, \bm x_{d}) )}\id (  \bm y  \in [\bm x_0, \dots, \bm x_{d} ]) d\bm x_0 \dots d\bm x_{d}\, d\bm y
.
\end{align*}
Here the last step follows from the multivariate Mecke formula \eqref{eq:Mecke}. Next, we use the Blaschke-Petkantschin integral transformation \eqref{eq:Bla}, which results in
\begin{align*}
\E \l_d ( A_{\eta_t}\setminus A)
& =
\frac{2}{d+1} t^{d+1}  \int\limits_{(\R^d)^2} \id(\bm y  \in A^c)  \id (\bm c  \in A)  \int\limits_0^\infty \int\limits_{(S^{d-1})^{d+1}} e^{- t r^d \kappa_d} r^{d^2-1}
\\ & \times
\id \left(  \bm y  \in [\bm c +r \bm u_0, \dots, \bm c +r \bm u_{d} ]\right)  \l_d([\bm u_0, \dots ,\bm u_{d}]) d\bm u_0 \dots d \bm u_{d}\,dr\, d\bm c  d\bm y .
\end{align*}
Substituting $r= t^{- \frac 1d} s$ and $\bm y = \bm c + t^{- \frac 1d} sx \bm u$ by new variables $s,x \geq 0,\, \bm u \in S^{d-1}$ gives
\begin{align*}
\E \l_d (A \Delta A_{\eta_t})
& =
\frac{2}{d+1} 
\int\limits_{S^{d-1}} \int\limits_0^\infty \int\limits_0^\infty  \int\limits_{\R^d} 
\id\left(\bm c \in (A^c-  t^{- \frac 1d} sx \bm u) \cap A\right)  d\bm c \  e^{- s^d \k_d}  s^{d^2+d-1} 
\\ & \times 
\int\limits_{(S^{d-1})^{d+1}}  \id (x \bm u \in [ \bm u_0, \dots, \bm u_{d} ])\l_d([\bm u_0, \dots ,\bm u_{d}]) d\bm u_0 \dots d\bm u_{d}\, x^{d-1}  ds dx d\bm u
\\ & =
\frac{2}{d+1} 
\int\limits_{S^{d-1}} \int\limits_0^\infty \int\limits_0^\infty  
\l_d\left((A^c-  t^{- \frac 1d} sx \bm u) \cap A\right)  \,  e^{- s^d \k_d}  s^{d^2+d-1} 
\\ & \times 
\int\limits_{(S^{d-1})^{d+1}}  \id (x \bm u \in [ \bm u_0, \dots, \bm u_{d} ])\l_d([\bm u_0, \dots ,\bm u_{d}]) d\bm u_0 \dots d\bm u_{d}\, x^{d-1}  ds dx d\bm u
.
\end{align*}
Here the covariogram of $A$ occurs in the integrand,
$$ \l_d ((A^c-  t^{- \frac 1d} sx \bm u ) \cap A)
=
g_A(0)- g_A(-t^{- \frac 1d} sx \bm u).$$
Galerne's Theorem \ref{thm:per} tells us that
\begin{align*}
	\lim_{t\to\infty}\frac{\l_d ((A^c-  t^{- \frac 1d} sx \bm u ) \cap A)}{t^{- \frac 1d} sx}=
	- \frac{\partial g_A}{\partial \bm u} (0),
\end{align*}
and that $g_A$ is Lipschitz with a bounded Lipschitz constant. Hence, we can apply Lebesgue's dominated convergence theorem, which leads to
\begin{align*}
\lim_{t \to \infty} t^{\frac 1d}\,
\E \l_d (A \Delta A_{\eta_t})
& =
- \frac{2}{d+1} 
\int\limits_{S^{d-1}} \int\limits_0^\infty \int\limits_0^\infty  
\frac{\partial g_A}{\partial \bm u} (0) \,  e^{- s^d \k_d}  s^{d^2+d} 
\\ & \hskip-1cm \times 
\int\limits_{(S^{d-1})^{d+1}}  \id (x \bm u \in [ \bm u_0, \dots, \bm u_{d} ])\l_d([\bm u_0, \dots ,\bm u_{d}]) d\bm u_0 \dots d\bm u_{d}\, x^{d}  ds dx d\bm u
.
\end{align*}
Because the spherical measure in the inner integral is rotation invariant, we can apply equation \eqref{eq:2} from Theorem \ref{thm:per} and obtain
\begin{align}\label{eq:symmdiff}
\lim_{t \to \infty} 
& \nonumber
t^{\frac 1d}\, \E \l_d (A \Delta A_{\eta_t})
= 
 \frac{2}{d+1} \k_{d-1} \Per(A)
\int\limits_0^\infty \int\limits_0^\infty
e^{- s^d \k_d}  s^{d^2+d} \, ds
\\ &  \nonumber \  \times
\int\limits_{(S^{d-1})^{d+1}}  \id (x \bm e_1 \in [ \bm u_0, \dots, \bm u_{d} ])\l_d([\bm u_0, \dots ,\bm u_{d}]) d\bm u_0 \dots d\bm u_{d}\, x^{d}  dx
\\ &=  \nonumber
\frac{2}{d(d+1)} \k_{d-1} \k_d^{- d-1 - \frac 1d } \Gamma\left(d+1+\frac 1d\right)  \Per(A)
\\ & \nonumber \  \times
\int\limits_0^{\infty} \int\limits_{(S^{d-1})^{d+1}}
\id (x \bm e_1 \in [ \bm u_0, \dots, \bm u_{d} ])  \l_d([\bm u_0, \dots , \bm u_{d}]) d\bm u_0 \dots d\bm u_{d}\, x^{d} dx
\\& =
c_d  \Per(A).
\end{align}
Thus, the expectation of the symmetric difference is of order $t^{-\frac 1d}$.

\medskip
\textit{Remark:}
The same order $t^{-\frac 1d}$ for the error of the symmetric difference occurs for the Poisson Voronoi Approximation, as shown in \cite{HevRei}\cite{ReiSpoZap}. Thus, it is of interest to compute the constant $c_d$ in \eqref{eq:symmdiff} and to compare it to the analogous constant of the Poisson-Voronoi approximation.

First, observe that $ x \bm e_1 \in [ \bm u_0, \dots, \bm u_{d} ]$ with $\bm u_i \in S^{d-1}$ implies  that $x\le 1$ and thus $x^d\le x^{d-1}$. This allows us to introduce in the integral above polar coordinates, using again the rotation invariance of the spherical measure. We obtain
\begin{align*}
&
\int\limits_0^\infty \int\limits_{(S^{d-1})^{d+1}}  \id (x \bm e_1 \in [ \bm u_0, \dots, \bm u_{d} ])\l_d([\bm u_0, \dots ,\bm u_{d}]) d\bm u_0 \dots d\bm u_{d}\, x^{d}  dx
\\ & \leq
\frac{1}{d\kappa_d} \int\limits_{S^{d-1}} \int\limits_0^{\infty} \int\limits_{(S^{d-1})^{d+1}}
\id (x \bm u \in [ \bm u_0, \dots, \bm u_{d} ])  \l_d([\bm u_0, \dots , \bm u_{d}]) d\bm u_0 \dots d\bm u_{d}\, x^{d-1} dx d\bm u
\\ & =
\frac{1}{d\kappa_d}  \int\limits_{(S^{d-1})^{d+1}} \int\limits_{B_d}
\id (\bm x \in [ \bm u_0, \dots, \bm u_{d} ])  d\bm x \ \l_d([\bm u_0, \dots , \bm u_{d}]) d\bm u_0 \dots d\bm u_{d}
\\ & =
\frac{1}{d\kappa_d}  \int\limits_{(S^{d-1})^{d+1}} \l_d([\bm u_0, \dots , \bm u_{d}])^2 d\bm u_0 \dots d\bm u_{d}
= \frac{1}{d\kappa_d}  S(d,d,2)
.
\end{align*}
Since $S(d,d,2)=\frac{d+1}{(d-1)!}\kappa_d^{d+1}$, we obtain for the constant $c_d$ in \eqref{eq:symmdiff}
\begin{equation*}
c_d
\le
2 d^{-2} \k_{d-1} \k_d^{-1 - \frac 1d }
\frac{\Gamma\left(d+1+\frac 1d\right)}{(d-1)!}
.
\end{equation*}
A similar computation and the use of the reverse H\"older inequality leads to a lower bound for $c_d$, 
$$
2 d^{-2} \k_{d-1} \k_d^{-1 - \frac 1d }\frac{\Gamma\left(d+1+\frac 1d\right)}{(d-1)!} \left[\frac 1{d \k_d} \frac{S(d,d,2)}{S(d,d,1)} \right]^{\frac {1}{d-1}}
\leq
c_d 
.  $$

\medskip
In contrast to the Poisson-Delaunay approximation, for the Poisson-Voronoi approximation the analogous dimension-dependent constant $c_d$ occurring in the formula for the symmetric difference could be calculated exactly, $c_d=2d^{-2}\kappa_{d-1}\kappa_{d}^{-1-1/d}\Gamma(1/d)$, see \cite{HevRei}\cite{ReiSpoZap}.
Because an elementary calculation shows that
$
\Gamma\left(d+1+\frac 1d\right)
\le e\Gamma\left(\frac{1}{d}\right) (d-1)!
$, which leads to 
\begin{equation*}
c_d
\le
2 e d^{-2} \k_{d-1} \k_d^{-1 - \frac 1d }
\Gamma\left(\frac{1}{d}\right),
\end{equation*}
the upper bound for $c_d$ for the Poisson-Delaunay approximation is larger by the factor $e $ than the constant for the Poisson-Voronoi approximation. It is unclear whether the difference is only an artefact of the method of approximation.

\subsection*{Acknowledgement}
Part of this work was done during a stay of MR at the Trimester Program \textit{Synergies between modern probability, geometric analysis and stochastic geometry} at the Hausdorff Research Institute for Mathematics (HIM) in Bonn (Germany).

\vspace{1cm}

\footnotesize

\textsc{Matthias Reitzner:} Institut f\"ur Mathematik, Universit\"at Osnabr\"uck, Germany \\
\textit{E-mail}: \texttt{matthias.reitzner@uni-osnabrueck.de}

\bigskip

\textsc{Anna Strotmann:} Institut f\"ur Mathematik, Universit\"at Osnabr\"uck, Germany \\
\textit{E-mail}: \texttt{anna.strotmann@uni-osnabrueck.de}

\end{document}